\documentclass{amsart}

\newtheorem{theorem}{Theorem}[section]
\newtheorem{lemma}[theorem]{Lemma}
\newtheorem{proposition}[theorem]{Proposition}
\newtheorem{corollary}[theorem]{Corollary}
\newtheorem{claim}[theorem]{Claim}

\theoremstyle{definition}

\theoremstyle{remark}
\newtheorem{remark}[theorem]{Remark}

\numberwithin{equation}{section}

\begin{document}

\title{ Four-dimensional Einstein manifolds with sectional curvature bounded from above}

\author{Zhuhong Zhang}
\address{Department of Mathematics, South China Normal Univeristy, Guangzhou, P. R. China 510275}
\email{juhoncheung@sina.com}






\begin{abstract}
Given an Einstein structure with positive scalar curvature on a  four-dimensional Riemannian manifolds, that is $Ric=\lambda g$ for some positive constant $\lambda$. For convenience, the Ricci curvature is always normalized to $Ric=1$. A basic problem is to classify four-dimensional Einstein manifolds with positive or nonnegative curvature and $Ric=1$. In this paper, we firstly show that if the sectional curvature satisfies $K\le M_1= \frac{\sqrt{3}}2\approx 0.866025$, then the sectional curvature will be nonnegative. 
Next, we prove a family of rigidity theorems of Einstein four-manifolds with nonnegative sectional curvature, and satisfies $K_{ik}+sK_{ij}\ge K_s=\frac{1+\sqrt{2}}3 - \frac{\sqrt{4+2\sqrt{2}}}4+\frac{2-\sqrt{2}}6 s$ for every orthonormal basis $\{e_i\}$ with $K_{ik}\ge K_{ij}$, where $s$ is any nonnegative constant.
Indeed, we will show that  these Einstein manifolds must be isometric either $S^4$, $RP^4$ or $CP^2$ with standard metrics. 
As a corollary, we give a rigidity result of Einstein four-manifolds with $Ric=1$, 
and the sectional curvature satisfies $K \le M_2=\frac {2-\sqrt{2}}6+\frac{\sqrt{4+2\sqrt{2}}}4 \approx 0.750912$. 
\end{abstract}

\maketitle



\section{Introduction}

A Riemannian manifold $(M, g)$ is called Einstein if the Ricci curvature satisfies the Einstein equation $$Ric=\lambda g$$ for some constant $\lambda$.
In differential geometric, a basic problem is to classify Einstein manifolds with positive or nonnegative sectional curvature in the category of either topology, diffeomorphism, or isometry.
By the work of Berger\cite{Be}, four-dimensional Einstein manifold with nonnegative sectional curvature and positive scalar curvature must have its Euler characteristic $\chi(M)$ bounded by $1\le \chi(M)\le 9$. Furthermore, Hitchin \cite{Hi} had shown that $|\tau(M)|\le (\frac 23)^{3/2}\chi(M)$, where $\tau(M)$ is the signature. Later, Gursky and Lebrun \cite{GL} improved this result, showed that $\chi(M)>\frac{15}4|\tau(M)|$ if the manifold is not half conformally flat. 

These results suggest that there are few four-dimensional manifolds can carry Einstein structure with nonnegative sectional curvature and positive scalar curvatur. Indeed,
Up to now, the only known examples of oriented four-dimensional Einstein manifolds with nonnegative sectional curvature and positive scalar curvature are the sphere $S^4$,  the product of 2-spheres $S^2\times S^2$, and the complex projective space $CP^2$.

In the isometric category, Hitchin's classification theorem (see \cite{Besse} Theorem 13.30) states that half conformally flat Einstein four-manifolds with positive scalar curvature are isometric to either $S^4$ or $CP^2$. If Einstein manifolds have positive curvature, Tachibana\cite{Ta} proved that Einstein manifolds with positive curvature operator are space forms. Recently, Brandle\cite{Br} improved this result and showed that Einstein manifolds with positive isotropy curvature are space forms, and with nonnegative isotropy curvature are locally symmetric.

On the other hand, If Einstein manifolds have positive or nonnegative sectional curvature, Gursky and LeBrun \cite{GL} showed that compact Einstein four-manifolds of nonnegative sectional curvature and positive intersection form are $CP^2$.
While Yang \cite{Yang} considered Einstein four-manifold with $Ric=1$, and with nonnegative sectional curvature. If the sectional curvature further satisfies condition $(a)$: $K\ge (\sqrt{1249}-23)/120 \approx 0.102843$ or condition $(b)$: $2K_{ik}+K_{ij}\ge \frac 9{14}$ for every orthonormal basis $\{e_i\}$ with $K_{ik}\ge K_{ij}$, then it must be isometric either $S^4$, $RP^4$ or $CP^2$ with standard metrics. Later, Costa \cite{Co} improve Yang's result, showed that if $K\ge \epsilon_0 = (2-\sqrt{2})/6 \approx 0.097631$, then Yang's result remains true.

In this paper, we will consider Einstein four-manifolds with sectional curvature bounded from above by a constant less than $1$.  
Firstly, we obtain a important observation of Einstein four-manifolds as follow.

\begin{theorem}
Let $(M, g)$ be a complete four-dimensional Einstein manifold with $Ric=1$, and the sectional curvature satisfies
$$K \le M_1  = \frac{\sqrt{3}}2\approx 0.866025.$$ 
Then, $(M, g)$ must have nonnegative sectional curvature.
\end{theorem}

In the next part, we consider the rigidity of Einstein four-manifolds with nonnegative sectional curvature. 

\begin{theorem}
For any constant $s \ge 0$, let $K_s=\frac{1+\sqrt{2}}3 - \frac{\sqrt{4+2\sqrt{2}}}4+\frac{2-\sqrt{2}}6 s$, and we have the following property.
Suppose $(M, g)$ be a complete four-dimensional Einstein manifold with nonnegative sectional curvature, and $Ric=1$.
If we further assume that 
$$K_{ik}+sK_{ij}\ge K_s$$
 for every orthonormal  basis $\{e_i\}\subset T_o M$, which satisfies $K_{ik}\ge K_{ij}$.
Then, $(M, g)$ must be isometry to either the Euclidean sphere $S^4$, the real projective space $RP^4$ with constant sectional curvature $K=\frac 13$, or the complex projective space $CP^2$ with the normalized Fubini-Study metric.
\end{theorem}

Take $s=\frac 12$, then we have the following rigidity theorem, which generalize Yang's result \cite{Yang}.

\begin{corollary}
Let $(M, g)$ be a complete four-dimensional Einstein manifold with nonnegative sectional curvature, and $Ric=1$.
If we further assume that 
$$2K_{ik}+K_{ij}\ge  \frac{2+\sqrt{2}}2- \frac{\sqrt{4+2\sqrt{2}}}4\approx 0.400543 $$
 for every orthonormal  basis $\{e_i\}\subset T_o M$, which satisfies $K_{ik}\ge K_{ij}$.
Then, $(M, g)$ must be isometry to either the Euclidean sphere $S^4$, the real projective space $RP^4$ with constant sectional curvature $K=\frac 13$, or the complex projective space $CP^2$ with the normalized Fubini-Study metric.
\end{corollary}

\begin{remark}
In \cite{Yang}, Yang showed the rigidity theorem under the condition that  $2K_{ik}+K_{ij}\ge \frac 9{14}\approx 0.642857$. 
\end{remark}

Combine the above two theorems, we obtain the following rigidity theorem of Einstein four-dimensional with sectional curvature bounded from above.
\begin{theorem}
Let $(M, g)$ be a complete four-dimensional Einstein manifold with $Ric=1$, and the sectional curvature $K$ satisfies
$$K \le M_2=\frac {2-\sqrt{2}}6+\frac{\sqrt{4+2\sqrt{2}}}4 \approx 0.750912.$$ 
Then, $(M, g)$ must be isometry to either the Euclidean sphere $S^4$, the real projective space $RP^4$ with constant sectional curvature $K=\frac 13$, or the complex projective space $CP^2$ with the normalized Fubini-Study metric.
\end{theorem}

\begin{remark}
In \cite{Co}, Costa showed a similar rigidity theorem under the condition that the sectional curvature satisfies $K \le \frac 23$. Our theorem can be considered as a generalization of Costa's result.
\end{remark}

A curial idea on the classification theorem of Einstein four-manifolds is the Weitzenb\"{o}ak formula. Usually one need to prove a Kato type inequality to get some curvature estimates for the self-dual and anti-dual Weyl tensor $W^{\pm}$ in terms of the length function $|W^{\pm}|$. 

However, our proof will base on the Ricci flow theory. In 1982, Hamilton \cite{Ha82} introduced Ricci flow to study compact three-manifolds with positive Ricci curvature. Later, Ricci flow theory have become a important tool in differential geometry. For example, by using the Ricci flow, we obtained the classification theorem of manifolds with positive curvature operator \cite{Ha86, BW}, Poincar$\acute{e}$ and geometrization conjecture \cite{P1, P2, CZ}, $1/4$-pinched differential sphere theorem \cite{BS1}, the classification theorem of manifolds with positive isotropy curvature \cite{Ha97, CZ05, CTZ}, etc.. 

In this paper, we will use  the advanced maximum principle to construct some pinching sets that invariant under the Ricci flow, which suggest that the curvature will become better and better along the Ricci flow. But since manifolds we consider are Einstein, the metric only change by scaling along the Ricci flow, so the initial metric must be good enough, and we get our rigidity theorem.

The rest of the paper is organized into five sections. In Section 2 , we introduce some basic facts about Einstein four-manifolds that will be used throughout the paper, and obtain a key estimate Lemma 2.2. We will introduce sone basic facts and ODE system about Ricci flow on four-manifolds in Section 3, and then we prove Theorem 1.1 by using Lemma 2.2. In Section 4 and Section 5, we will prove rigidity results Theorem 1.2 and Theorem 1.5, respectively. The arguments are base on the pinched estimate Claim 4.2, which will be proved by using a key estimate Lemma 4.1.

 {\bf Acknowledgements} The author was partially supported by NSFC 11301191.

\section{Preliminaries}

Let $(M, g)$ be a closed oriented four-dimensional Einstein manifold. The space $\Lambda^2_\pm$ of self-dual and anti-self dual 2-forms on $M$ are the eigenspaces of the eigenvalues $+1$ and $-1$ of  the Hodge star operator $\ast$ on 2-forms, respectively. This gives an orthogonal decomposition 
$$\Lambda^2=\Lambda^2_+\oplus\Lambda^2_-.$$
Furthermore, this decomposition induces  a block decomposition of the curvature operator matrix as 
$$
 M_{\alpha\beta}=   \left (
       \begin{array}{lll}
       A \ & B\\[1mm]
       ^tB \ & C \\[1mm]
       \end{array}
    \right),
$$
and  $B\equiv 0$ on $M$ since $M$ is Einstein. 

For any point $o\in M$, we will denote by $K(\pi)$ the sectional curvature of the plane $\pi\subset T_oM$.  It is well known that $$K(\pi)=K(\pi^{\bot}),$$
where $\pi^{\bot}\subset T_oM$ is the plane perpendicular to $\pi$. More precisely, by choosing a positive oriented orthonormal basis $\{e_i\}$ of $T_o M$, we have $K_{12}=K_{34}$, $K_{13}=K_{24}$ and $K_{14}=K_{23}$, where $K_{ij}$ is the sectional curvature of the plane spanned by $e_i, e_j$.

Let  $\{\theta_i\}$ be the dual orthonormal coframe, then a basis of $\Lambda^2_+$ is
$$\varphi_1=\theta_1\wedge \theta_2 + \theta_3 \wedge \theta_4, \quad \varphi_2=\theta_1\wedge \theta_3 + \theta_4 \wedge \theta_2,\quad \varphi_3=\theta_1\wedge \theta_4 + \theta_2 \wedge \theta_3,$$
while a basis for $\Lambda^2_-$ is
$$\psi_1=\theta_1\wedge \theta_2 - \theta_3 \wedge \theta_4,\quad \psi_2=-\theta_1\wedge \theta_3 + \theta_4 \wedge \theta_2,\quad \psi_3=\theta_1\wedge \theta_4 + \theta_2 \wedge \theta_3.$$

In order to give a representation more specific for the curvature operator matrix $A$ and $C$, we shall need the following lemma due to Berger \cite{Be}.

\begin{proposition}
Let $(M, g)$ be an oriented four-dimensional Einstein manifold. For a point $o\in M$, there exists a positive oriented orthonormal basis $\{e_i\}$ of $T_o M$, such that the curvature tensor $\{R_{ijkl}\}$ satisfies the following properties. \\
(1). $K_{12}=R_{1212}=\min\{K(\pi)| \pi\subset T_o M\}$.   \\
(2). $K_{14}=R_{1414}=\max\{K(\pi)| \pi\subset T_o M\}$.   \\
(3). $R_{ikjk}=0$ for all $i\neq j$. \\
(4). $|R_{1342}-R_{1234}|\le K_{13}-K_{12}$, $|R_{1423}-R_{1342}|\le K_{14}-K_{13}$, and $|R_{1423}-R_{1234}|\le K_{14}-K_{12}$. 
\end{proposition}

Throughout this paper, we always choose the above basis $\{e_i\}$ of $T_oM$ to get some estimates. A simple fact is the following observation of curvature operator matrix.
For the above basis $\{e_i\}$ of $T_oM$, we obtain the dual coframe $\{\theta_i\}$, and then a basis  $\{\varphi_i\}$ of $\Lambda^2_+$ ,  a basis  $\{\psi_i\}$ of $\Lambda^2_-$.

A direct computation shows that
$
 A=   \left (
       \begin{array}{lll}
       a_1 \ &0     \ &0    \\[1mm]
       0     \ &a_2 \ &0    \\[1mm]
       0     \ &0     \ &a_3\\[1mm]
       \end{array}
    \right),
$
and
$
 C=   \left (
       \begin{array}{lll}
       c_1 \ &0     \ &0    \\[1mm]
       0     \ &c_2 \ &0    \\[1mm]
       0     \ &0     \ &c_3\\[1mm]
       \end{array}
    \right),
$
where
$a_1 = 2(K_{12}+R_{1234})
\le a_2 = 2(K_{13}+R_{1342})
\le a_3 = 2(K_{14}+R_{1423})$
and
$c_1 = 2(K_{12}-R_{1234})
\le c_2 = 2(K_{13}-R_{1342})
\le c_3 = 2(K_{14}-R_{1423})$.

Define a quadratic function of curvature operator as follow
$$I = (c_2-c_1)c_3+(c_3-c_1)c_2+(a_2-a_1)a_3+(a_3-a_1)a_2.$$

Now we can pose our key lemma in this section.

\begin{lemma} 
Let $(M, g)$ be a complete oriented four-dimensional Einstein manifold  with $Ric=1$, and with the sectional curvature bounded from above by $K\le \frac{\sqrt{3}}2$.
Fixed a point $o\in M$, if at this point the least sectional curvature $K_{12} \le -\epsilon< 0$ for some positive constant $\epsilon$. Then the following hold.

(1). If $a_2\ge 0$ and $c_2\ge 0$, then 
$$ I \ge \frac{16}3 \epsilon .$$

(2). If $a_2<0$ or $c_2<0$, then $$I >\frac 14 \epsilon .$$

\end{lemma}

\begin{proof} Choose the basis $\{e_i\}$ of Proposition 2.1. Note that $K_{12}+K_{13}=1-K_{14}\ge 1-\frac{\sqrt{3}}2>0$, the manifold has two-positive sectional curvature.

(1). Since $a_2 \ge 0$ and $ c_2\ge 0$, we have $I\ge  (c_2-c_1)c_3+(a_2-a_1)a_3$.

Now $Ric=1$, so $a_1+a_2+a_3=c_1+c_2+c_3=\frac R2=2$. 
And then $a_3\ge \frac 13 (a_1+a_2+a_3)= \frac 23$. Similarly, $c_3\ge \frac23$. So we can obtain
\begin{align*}
I\ge& \frac23 [(a_2+c_2) - (a_1+c_1)]  \\
=&\frac83(K_{13}-K_{12})=\frac83(K_{12}+K_{13}-2K_{12})  \\
>& \frac{16}3 \epsilon.
\end{align*}

(2).  Without lose of generality, we can assume $a_2<0$. By direct computation, we have
\begin{align*}
\frac 14 I =& [ ( K_{13} - R_{1342} )  - ( K_{12} - R_{1234} )  ]   ( K_{14} - R_{1423} )     \\
                 & +  [ ( K_{13} + R_{1342} )  - ( K_{12} + R_{1234} )  ]   ( K_{14} + R_{1423} )     \\
                 & + [ ( K_{14} - R_{1423} )  - ( K_{12} - R_{1234} )  ]   ( K_{13} - R_{1342} )     \\
                 & +  [ ( K_{14} + R_{1423} )  - ( K_{12} + R_{1234} )  ]   ( K_{13} + R_{1342} )     \\
               =& 2(K_{13}-K_{12}) K_{14} + 2 (R_{1342} - R_{1234})R_{1423}  \\
                 & + 2(K_{14}-K_{12}) K_{13} + 2 (R_{1423} - R_{1234})R_{1342}.
\end{align*}

Hence by the first Bianchi identity, 
$$\frac 18 I = (K_{13} -K_{12})K_{14} + (K_{14}-K_{12}) K_{13}  + R_{1234}^2 +  2 R_{1342} R_{1423} . $$

The condition of two-positive sectional curvature implies that $K_{13}>-K_{12}\ge\epsilon>0$. It is easy to see that $K_{14}\ge \frac{K_{14}+K_{13}}2>\frac{K_{14}+K_{13}+K_{12}}2=\frac 12$. Denote by $x=-R_{1234}$, $y=-R_{1342}$. Then $R_{1423}=x+y$, and hence
$$R_{1234}^2 +  2 R_{1342} R_{1423} = x^2-2y(x+y).$$

In the following, we divide the argument into three cases.

{\bf Case 1:  $K_{14}-K_{13}\le \frac72(K_{13}-K_{12})$ . }

In this case, $K_{14}-K_{13}<7K_{13}$, and then $8K_{13}>K_{14}$.

Note that
\begin{align*}
R_{1234}^2 +  2 R_{1342} R_{1423}
=& \frac 13 [ 2(x-y)^2 + 2(x-y)(x+2y) - (x+2y)^2 ]\\
\ge&  -\frac 12(x+2y)^2 .\\
\end{align*}

Now since $a_2<0$, we obtain $a_1<0$, and hence $$x>K_{12},\ y>K_{13}>0.$$ 
This implies that $ x+y>K_{12}+K_{13}  >0$, and $x+2y>0$. Furthermore, by Proposition 2.1, we have
$$ x+2y = R_{1423}-R_{1342} \le K_{14}-K_{13},$$

so we have  
\begin{align*}
\frac 18 I  
\ge& (K_{13} -K_{12})K_{14} + (K_{14}-K_{12}) K_{13}  -\frac 12 (K_{14}-K_{13})^2 \\
=&  (K_{13} -K_{12})(K_{14}+K_{13})+  (K_{14}-K_{13}) K_{13}  -\frac 12 (K_{14}-K_{13})^2 \\
>& (K_{13} -K_{12})(K_{14}+K_{13}) - (\frac 12-\frac 17) (K_{14}-K_{13})^2 \\
>& (K_{13} -K_{12})(K_{14}+K_{13}) - \frac 54(K_{13} -K_{12}) (K_{14}-K_{13}) \\
=& \frac 14(K_{13} -K_{12})(9K_{13}-K_{14})\\
>& \frac 14 \cdot 2(-K_{12})\cdot \frac 18K_{14}  >   \frac 1{32} \epsilon.
\end{align*}

{\bf Case 2:  $K_{14}-K_{13}> \frac72(K_{13}-K_{12})$, and $x\le 0$ . }

By Proposition 2.1, we have
$$ |y-x|\ = |R_{1342}-R_{1234}| \le K_{13}-K_{12}.$$

And then
\begin{align*}
R_{1234}^2 +  2 R_{1342} R_{1423}
=& 3x^2 -6xy- 2(x-y)^2 \\
\ge&  -2(K_{13}-K_{12})^2 .\\
\end{align*}

so we have  
\begin{align*}
\frac 18 I  
\ge& (K_{13} -K_{12})K_{14} + (K_{14}-K_{12}) K_{13}  -2 (K_{13}-K_{12})^2 \\
\ge& (K_{13} -K_{12})(K_{14}+K_{13}) + (K_{14}-K_{13}) K_{13}  -2 (K_{13}-K_{12})^2 \\
>& (K_{13}-K_{12}) [K_{14} + \frac 12K_{13} + 2(K_{13}+K_{12})] \\
>& 2(-K_{12}) K_{14} >\epsilon.
\end{align*}

{\bf Case 3:  $K_{14}-K_{13}> \frac72(K_{13}-K_{12})$, and $x> 0$ . }

Similarly, we have $x+2y  \le K_{14}-K_{13}$,  $y>0$, and
\begin{align*}
R_{1234}^2 +  2 R_{1342} R_{1423}
=& \frac 13 [ 2(x-y)^2 + 2(x-y)(x+2y) - (x+2y)^2 ]\\
\ge& \frac 13 [ 2(x-y)^2 + 2(x-y)(K_{14}-K_{13}) - (K_{14}-K_{13})^2 ].\\
\end{align*}

Furthermore, by $ |y-x| \le K_{13}-K_{12}$, and
$$x-y >  -\frac 12(x+2y) \ge -\frac 12 (K_{14}-K_{13}) ,$$
so we have
$$R_{1234}^2 +  2 R_{1342} R_{1423}
\ge \frac 13 [ 2(K_{13}-K_{12})^2 - 2(K_{13}-K_{12})(K_{14}-K_{13}) - (K_{14}-K_{13})^2 ].  $$

Denote by $m=K_{12}$, $M=K_{14}$, then $z=K_{13}=1-(m+M)$, and
\begin{align*}
\frac 38 I 
=& 3M(z-m) +3z(M-m) + [ 2(z-m)^2 -2 (z-m)(M-z) -(M-z)^2 ]\\
=& 3M(1-M-2m) + 3(1-M-m)(M-m)  + 2( 1-M-2m  )^2\\
  & -2 (1-M-2m)(2M+m-1) - (2M+m-1)^2  \\
=& -4M^2 + 3 + 8mM -15m+14m^2\\
=& -4M^2 + 3 + (-m)(15-8M )+14m^2\\
>& 7\epsilon,
\end{align*}
the last inequality holds because $-4M^2+3\ge0$.

Combining the above argument, we complete the proof of Lemma 2.2.

\end{proof}

\section{Einstein four-manifolds with two-positive sectional curvature}

Let $(M, g)$ be a closed Riemannian manifold, and $g_{ij}(t)$ is the unique short time solution  of Ricci flow 
$$\frac{\partial}{\partial t}g_{ij}(t)=-2R_{ij}(t).$$ 
In terms of moving frames \cite{Ha86},  the curvature operator $M_{\alpha\beta}$ evolves by
$$\frac{\partial}{\partial t}M_{\alpha\beta}=\triangle M_{\alpha\beta}+M_{\alpha\beta}^2+M_{\alpha\beta}^{\#} ,$$
where $M_{\alpha\beta}^{\#}$ is the Lie algebra adjoint of $M_{\alpha\beta}$. And the ODE corresponding to the above equations is 
$$\frac{d}{d t}M_{\alpha\beta}=M_{\alpha\beta}^2+M_{\alpha\beta}^{\#} .\eqno(\rm{ODE})$$

Suppose $(M, g)$ is a closed oriented four-dimensional normalized Einstein manifold with positive scalar curvature, it is easy to see that the unique solution of Ricci flow is  a self-similar solution given by 
$$g(t)=(1-2 t)g,\quad t\in (-\infty, \frac12).$$

Follow by the discussion in the last section, we have a good block decomposition of the curvature operator matrix
$$
 M_{\alpha\beta}=   \left (
       \begin{array}{lll}
       A \ & 0\\[1mm]
       0 \ & C \\[1mm]
       \end{array}
    \right),
$$
where $A$ and $C$ are diagonalized by the eigenvalues $\{a_i\}$ and $\{c_i\}$, respectively.
Furthermore, the ODE corresponding to the Ricci flow of $A$ and $C$ becomes
$$\frac{d}{dt}A=A^2+2A^{\sharp}, \qquad \frac{d}{dt}C=C^2+2C^{\sharp}.$$
So the ODE system of eigenvalues are given by
$$
  \left \{
       \begin{array}{lll}
       \frac{d}{dt}a_1 &=a_1^2 + 2a_2a_3, \\[2mm]
       \frac{d}{dt}a_2 &=a_2^2 + 2a_1a_3, \\[2mm]
       \frac{d}{dt}a_3 &=a_3^2 + 2a_1a_2, \\[2mm]
       \end{array}
    \right\{
       \begin{array}{lll}
       \frac{d}{dt}c_1 &=c_1^2 + 2c_2c_3, \\[2mm]
       \frac{d}{dt}c_2 &=c_2^2 + 2c_1c_3, \\[2mm]
       \frac{d}{dt}c_3 &=c_3^2 + 2c_1c_2, \\[2mm]
       \end{array}.
$$

Now we can prove Theorem 1.1.

\begin{proof}{\bf of Theorem 1.1}
If $M$ is not orientable, we can lift the Einstein metric onto an oriented 2-cover of $M$. So we always 
assume $(M, g)$ is an oriented four-dimensional Einstein manifolds with $Ric=1$, and $K \le \frac{\sqrt{3}}2$. 

Let $\kappa=\min \frac KR$, where $R$ is the scalar curvature. We only need to prove $\kappa\ge0$. If not, then there exist some constant $\epsilon>0$, such that $K_{12}=\min{K(\pi)}<-\epsilon$. Now for the self-similar solution of Ricci flow $g(t)$, we have the following claim.
\begin{claim}
There exist a constant $\delta>0$, such that $$a_1+c_1\ge (\kappa+\delta t)R$$ is preserved under the Ricci flow for all $t\ge 0$.
\end{claim}

\emph{proof of Claim 3.1.} Consider the set $\Omega$ of matrices defined by the above inequality. It is easy to see that $\Omega$ is closed, convex and $O(n)$-invariant. By using the advanced maximum principle, we only need to show the set $\Omega$ is preserved by the (ODE) system. Indeed, we only need to look at points on the boundary of the set.

Suppose at some point $(o, t)$, $a_1+c_1= (\kappa+\delta t)R$. Then
\begin{align*}
\frac d{dt}(a_1+c_1)
=&a_1^2+2a_2a_3 + c_1^2+2c_2c_3 \\
=& a_1(a_1+a_2+a_3)+c_1(c_1+c_2+c_3)+I \\
=& (a_1+c_1)\cdot \frac R2 + I \\
=& (\kappa+\delta t) \frac{R^2}2 + I \\
\end{align*}

On the other hand, since $g(t)$ is just a scaling of Einstein matric $g$, $R_{ij}(t)=\lambda(t) g_{ij}(t)=\frac {R(t)}4 g_{ij}(t)$,
\begin{align*}
\frac d{dt}  \Big[ (\kappa+\delta t)R \Big] 
=& (\kappa+\delta t) \frac d{dt}R + \delta R \\
=&(\kappa+\delta t) \cdot 2|Ric|^2 + \delta R \\
=&(\kappa+\delta t) \cdot \frac {R^2}2 + \delta R \\
\end{align*}

Furthermore, follow by Lemma 2.2,
$$I > C(\epsilon)R$$
for some positive constant $C=C(\epsilon)$. Take $\delta=C(\epsilon)$, and we complete the proof of Claim 3.1.

By Claim 3.1, we obtain that 
$$a_1+c_1> \kappa R$$
for all small $t>0$, but this is impossible, since $g(t)$ is a scaling of $g$, and then $a_1+c_1= \kappa R$ at some point for all t.

\end{proof}

\section{Rigidity of Einstein four-manifolds with curvature bounded from below}

In this section, we will consider the rigidity of Einstein four-manifolds with curvature bounded from below. Given any constant $s \ge 0$, let $$K_s=\frac{1+\sqrt{2}}3 - \frac{\sqrt{4+2\sqrt{2}}}4+\frac{2-\sqrt{2}}6 s.$$ Then we can get the following key lemma.
\begin{lemma} 
Suppose $(M, g)$ is a complete oriented four-dimensional Einstein manifold  with nonnegative sectional curvature and $Ric=1$. If we further assume that $K_{ik} + sK_{ij} \ge K_s$ for every orthonormal basis $\{e_i\}$ with $K_{ik}\ge K_{ij}$, and the least sectional curvature $K_{12}\le \epsilon_0-\epsilon$, where $\epsilon_0=\frac{2-\sqrt{2}}6$. Then the following hold.

(1). If $a_2\ge 0$ and $c_2\ge 0$, then 
$$ I > \frac83 \epsilon. $$

(2). If $a_2<0$ or $c_2<0$, then $$I >\epsilon .$$

\end{lemma}

\begin{proof}
Choose the basis $\{e_i\}$ of Proposition 2.1. It is easy to see that
$$ K_{13}+sK_{12} < 1+\frac{2-\sqrt{2}}6 s < \frac{1+\sqrt{2}}3 + \frac{\sqrt{4+2\sqrt{2}}}4+\frac{2-\sqrt{2}}6 s.$$

(1). Follow the same argument as Lemma 2.2,  
\begin{align*}
I\ge&  (c_2-c_1)c_3+(a_2-a_1)a_3 \\
\ge& \frac23 [(a_2+c_2) - (a_1+c_1)]  \\
=&\frac83(K_{13}-K_{12})=\frac83 [ K_{13} + sK_{12} - (1+s)K_{12} ]  \\
\ge& \frac83 [K_s - (1+s)\epsilon_0+(1+s)\epsilon] \\
\ge& \frac83 [\frac{\sqrt{2}}2 - \frac{\sqrt{4+2\sqrt{2}}}4+(1+s)\epsilon] > \frac83 \epsilon.
\end{align*}

(2).  Similarly, we can assume $a_2<0$. Then $a_1<0$. 
Let $x=-R_{1234}$, $y=-R_{1342}$. We have $$x>K_{12}\ge 0,\ y>K_{13}>0.$$ 
Furthermore,
$$ 0<x+2y = R_{1423}-R_{1342} \le K_{14}-K_{13},$$
and
$$ |y-x|\ = |R_{1342}-R_{1234}| \le K_{13}-K_{12}.$$

Since
$$\frac 18 I = (K_{13} -K_{12})K_{14} + (K_{14}-K_{12}) K_{13}  + R_{1234}^2 +  2 R_{1342} R_{1423} ,$$
and
\begin{align*}
R_{1234}^2 +  2 R_{1342} R_{1423}
=& x^2-2y(x+y)  \\
=& \frac 13 [ 2(x-y)^2 + 2(x-y)(x+2y) - (x+2y)^2 ]\\
\ge& \frac 13 [ 2(x-y)^2 + 2(x-y)(K_{14}-K_{13}) - (K_{14}-K_{13})^2 ]\\
\end{align*}

Note that
$$x-y\ge -\frac 12(x+2y) \ge -\frac 12 (K_{14}-K_{13}) ,$$
so we have
$$R_{1234}^2 +  2 R_{1342} R_{1423}
\ge \frac 13 [ 2(K_{13}-K_{12})^2 - 2(K_{13}-K_{12})(K_{14}-K_{13}) - (K_{14}-K_{13})^2 ].  $$

Denote by $m=K_{12}$, $z=K_{13}+sK_{12}$, then $K_{14}=1-K_{12}-K_{13}$, and
\begin{align*}
\frac 38 I 
=& 3(K_{13} -K_{12})(1-K_{12}-K_{13}) +3 (1-2K_{12}-K_{13}) K_{13}\\
   & + 2(K_{13}-K_{12})^2 - 2(K_{13}-K_{12})(1-K_{12}-2K_{13}) - (1-K_{12}-2K_{13})^2  \\
=& -1+K_{12}+8K_{13}+2K_{12}^2-4K_{13}^2-16K_{12}K_{13}\\
=& -4z^2 + 8(1+sm-2m)z + [-1+(1-8s)m+2(1+8s-2s^2)m^2] \\
\end{align*}

It is easy to see that
\begin{align*}
\frac{\partial}{\partial m}(\frac 38 I)
=& 8z(s-2) +1-8s+4m(1+8s-2s^2)   \\
=&  -(16z-1-4m) -8s(1+ms-z-4m).
\end{align*}

Since $s$ is nonnegative, $z\ge K_s\ge  \frac{1+\sqrt{2}}3 - \frac{\sqrt{4+2\sqrt{2}}}4 \approx 0.151456$, and $m\le \epsilon_0 \approx 0.097631$, we have
$$16z-1-4m>1.$$
On the other hand,
$$1+ms-z-4m = K_{14}-3m\ge \frac 13-3\epsilon_0>0.$$
Hence $$\frac{\partial}{\partial m}(\frac 38 I)<-1.$$

So
$$\frac 38 I \ge \frac 38 I\Big|_{m=\epsilon_0-\epsilon} > \frac 38 I\Big|_{m=\epsilon_0} +\epsilon.$$

While
\begin{align*}
\frac 38 I\Big|_{m=\epsilon_0}=&-4z^2 + 8(1+s\epsilon_0-2\epsilon_0)z  \\
                                                 &+ [-1+(1-8s)\epsilon_0+2(1+8s-2s^2)\epsilon_0^2] \ge0,
\end{align*}
because the facts that $z=K_{13}+sK_{12}< \frac{1+\sqrt{2}}3 +\frac{2-\sqrt{2}}6 s +  \frac{\sqrt{4+2\sqrt{2}}}4$, and
\begin{align*}
z\ge&  1-2\epsilon_0+s\epsilon_0 - \frac 12 \sqrt{3(1-2\epsilon_0)(1-3\epsilon_0)}  \\
            = &  \frac{1+\sqrt{2}}3 +\frac{2-\sqrt{2}}6 s- \frac{\sqrt{4+2\sqrt{2}}}4.
\end{align*}

\end{proof}

Now we can following a similar argument to prove our main theorem 1.2.

\begin{proof}{\bf of Theorem 1.2}
Follow the same argument as Theorem 1.1,  we can 
assume $(M, g)$ is an oriented four-dimensional Einstein manifolds with nonnegative sectional curvature and $Ric=1$. 
In the following, we will divide the argument into two cases.

{\bf Case 1.}   $\min K\ge \epsilon_0$. In this case,  by the theorem of Costa \cite{Co}, the theorem holds.

{\bf Case 2.}  $\min K< \epsilon_0$. But this is impossible. We will argue by contradiction.
Since $M$ is compact, there exist a constant $\epsilon>0$, such that $\min K \le \epsilon_0-\epsilon$.

Let $\kappa=\min \frac KR$, where $R$ is the scalar curvature. We have the following assertion.

\begin{claim} There exist a constant $\delta>0$, such that $$a_1+c_1\ge (\kappa+\delta t)R$$ is preserved under the Ricci flow for all $t\ge 0$.
\end{claim}

\emph{proof of Claim 4.2.} Similarly, we only need to show that along the (ODE) system,
$$\frac{d}{dt}(a_1+c_1)\ge \frac{d}{dt}[(\kappa+\delta t)R]$$
at the boundary $a_1+c_1= (\kappa+\delta t)R$. Note that
$$\frac d{dt}(a_1+c_1) = (\kappa+\delta t) \frac{R^2}2 + I, $$
and
$$\frac d{dt}  \Big[ (\kappa+\delta t)R \Big]  =(\kappa+\delta t) \cdot \frac {R^2}2 + \delta R.$$

Then by Lemma 4.1,
$$I > C(\epsilon)R$$
for some positive constant $C=C(\epsilon)$. Take $\delta=C(\epsilon)$, and we get our assertion. 
\break

But Claim 4.2 develops a contradiction that 
$$a_1+c_1> \kappa R$$
for all small $t>0$. And we complete the proof of Theorem 1.2.

\end{proof}

\section{Rigidity of Einstein four-manifolds with two-positive sectional curvature}

In this section, we will consider the rigidity of Einstein four-manifolds with two-positive sectional curvature. 

Take $s=1$ in Lemma 4.1. Then the assumption of curvature becomes $K_{ik}+K_{ij}\ge K_1$, that means the sectional curvature satisfies
$$K\le 1-K_1=\frac {2-\sqrt{2}}6+\frac{\sqrt{4+2\sqrt{2}}}4=M_2.$$

Now we can prove our main theorem 1.5.

\begin{proof}{\bf of Theorem 1.5}
Similarly, we can assume $(M, g)$ is an oriented four-dimensional Einstein manifolds with $Ric=1$. And the sectional curvature has a upper bound $M_2$ means
 $$K_{ik}+K_{ij}\ge K_1.$$
On the other hand, $M_2<M_1$, so Theorem 1.1 implies that the sectional curvature must be nonnegative. 

Hence $\min K\ge \epsilon_0$. (Otherwise, we can follow the same argument as Theorem 1.2, by using Lemma 4.1 to construct a pinched set invariant under the Ricci flow, and develop a contradiction.)  So we can apply Costa's result \cite{Co} to get our rigidity theorem.

\end{proof}

\bibliographystyle{amsplain}

\end{document}